\documentclass[a4paper,11pt]{amsart}
\usepackage[utf8]{inputenc}
\usepackage{amsmath,amssymb}
\usepackage{array,enumerate}
\usepackage[left=3.5cm,right=3.5cm,bottom=3.5cm,top=3.5cm]{geometry}
\usepackage{comment,amsthm, mathtools}
\usepackage{natbib}
\usepackage{tikz-cd}
\usepackage{hyperref}

\bibpunct{[}{]}{,}{n}{}{;}

\newtheorem{theorem}{Theorem}[section]
\newtheorem{lemma}[theorem]{Lemma}
\newtheorem{proposition}[theorem]{Proposition}

\newtheorem{corollary}[theorem]{Corollary}
\theoremstyle{definition}
\newtheorem{definition}[theorem]{Definition}
\newtheorem{example}[theorem]{Example}

\DeclareMathOperator{\Aut}{Aut}

\DeclareMathOperator{\Sym}{Sym}

\newcommand{\SA}{\mathbb{A}}                    %
\newcommand{\CA}{\mathcal{A}}                    %
\newcommand{\CB}{\mathcal{B}}                    %
\newcommand{\CM}{\mathcal{M}}                    %
\newcommand{\CT}{\mathcal{T}} 
 
\newcommand{\CV}{\mathcal{V}}                    %
\mathchardef\mhyphen="2D

\allowdisplaybreaks[1]

\title{A power structure over the Grothendieck ring of geometric dg categories}
\author{\'Ad\'am Gyenge}
\address{Budapest University of Technology and Economics, Department of Algebra and Geometry, Institute of Mathematics, M\H{u}egyetem rakpart 3, H-1111, Budapest, Hungary}
\email{gyenge.adam@ttk.bme.hu}



\begin{document}

\begin{abstract}
We prove the existence of a power structure over the Grothendieck ring of geometric dg categories. 
We show that a conjecture by Galkin and Shinder (proved recently by Bergh, Gorchinskiy, Larsen, and Lunts) relating the motivic and categorical zeta functions of varieties can be reformulated as a compatibility between the motivic and categorical power structures. Using our power structure we show that the categorical zeta function of a geometric dg category can be expressed as a power with exponent the category itself. We give applications of our results for the generating series associated with Hilbert schemes of points, categorical Adams operations and series with exponent a linear algebraic group. 
\end{abstract}

\maketitle


\section{Introduction}
\label{sec:intro}

A power structure over a (semi)ring $R$
is a map \[(1+ t R[[t]]) \times R \to 1 + t R[[t]]\colon (A(t), m) \to (A(t))^m\]
 such that all the usual properties of the exponential map hold. A power structure over the Grothendieck ring of varieties $K_0(Var)$ was defined in \cite{gusein2004power}.
It has turned out to be an effective tool in expressing certain generating functions associated with varieties, see e.g. \cite{gusein2006power}. Power structures on rings have a deep connection with $\lambda$-ring structures. The ring $K_0(Var)$ has a well-known 1-$\lambda$-ring structure (or, shortly, a 1-$\lambda$-structure) induced by the symmetric powers of varieties \cite{gusein2004power}. It turns out that the motivic zeta function of a variety has a particularly nice expression as a power with exponent the motive of the variety.

The derived category of coherent sheaves on a variety has been proposed as an analogue of the
motive of a variety for a long time now \cite{bondal2002derived}. In this analogy semiorthogonal decompositions are the tools for the simplification of this
``motive'' similar to splitting by projectors in Grothendieck's motivic theory. It has since then turned out that it is better to work with dg enhancements of these triangulated categories. In \cite{bondal2004grothendieck} it was shown that there is a motivic measure, i.e. a ring homomorphism
\[ \phi \colon K_0(Var) \to  K_0(gdg\mhyphen cat)\]
to a certain Grothendieck ring $K_0(gdg\mhyphen cat)$ of geometric dg categories. Briefly, a dg category is geometric if its homotopy category is a semiorthogonal summand in the derived category of perfect complexes on some smooth projective variety.

On the other hand, symmetric powers on $K_0(gdg\mhyphen cat)$ as defined in \cite{ganter2014symmetric} do not induce a 1-$\lambda$-ring structure, but instead they induce a 2-$\lambda$-ring structure. For example, in a 1-$\lambda$-ring $\mathrm{Sym}^n(1)=1$ for any integer $n$, whereas in a 2-$\lambda$-ring $\mathrm{Sym}^n(1)=p(n)$, the number of partitions of $n$. Both 1-$\lambda$-rings and 2-$\lambda$-rings are special cases of  pre-$\lambda$-rings.  It was shown in \cite{gusein2013pre} and we recall it in Section~\ref{subsec:powstrprel} below that from each pre-$\lambda$-structure on a ring one can obtain a power structure. The converse is not true in general. Many pre-$\lambda$-structures may correspond to one and the same power structure. 

When applying the method of \cite{gusein2013pre} on $K_0(Var)$ (resp. on $K_0(gdg\mhyphen cat)$), one starts with the motivic (resp. categorical) zeta function of an object that collects the symmetric powers of the object into a generating function. With the 1-$\lambda$-ring (resp. 2-$\lambda$-ring) structure induced by this zeta function on $K_0(Var)$ (resp. on $K_0(gdg\mhyphen cat)$) the ring homomorphism $\phi$ cannot be a $\lambda$-ring homomorphism. 
Our main observation is that $\phi$ is nevertheless compatible with the induced power structures.
\begin{theorem}[{Theorem~\ref{thm:ringhomo2}}]
	\label{thm:mupower}
	Equip $K_0(Var)$ with the power structure described in Example~\ref{ex:K0var} and $K_0 (gdg\mhyphen cat)$ with the power structure  described in Example~\ref{ex:dgcat2lr}. 
	Then the ring homomorphism $\phi \colon K_0(Var) \to K_0 (gdg\mhyphen cat)$ is compatible with the power structures. That is,
	\[ \phi\left((A(t))^{m}\right)=(\phi(A(t)))^{\phi(m)}. \]
\end{theorem}

Galkin and Shinder proved in \cite{galkin2015zeta} the following relation between the motivic and the categorical zeta functions of quasi-projective varieties of dimensions 1 and 2, and they conjectured the same relationship to hold in any dimension:
\[ Z_{cat}(\phi(X),t)=\prod_{n \geq 1} \phi(Z_{mot}(X,t^n)). \]
The conjecture in full generality was proved recently by Bergh, Gorchinskiy, Larsen, and Lunts \cite{bergh2017categorical}.
It turns out that the compatibility of $\phi$ with the power structures on its domain and range  is essentially equivalent with the conjecture. Hence, Theorem \ref{thm:mupower} can also be understood as a reinterpretation of the Galkin--Shinder conjecture in terms of power structures. 
Given the several recent applications of power structures \cite{morrison2015motivic, g2022configuration, ricolfi2020virtual}, we believe that this reinterpretation may also have its own benefits. 

Compatibility of power structures for another motivic measure is a central topic in \cite{pajwani2023power, bejleri2024symmetric}. In these cases the measure has values in the Grothendieck-Witt ring of quadratic forms. These results together with our Theorem~\ref{thm:mupower} give the hint that there may be a deeper relationship between the Grothendieck ring of geometric dg categories and the Grothendieck-Witt ring of quadratic forms; this relationship can be the subject of future studies. In this direction, see also \cite{vezzosi2016quadratic} for a construction of the Grothendieck-Witt ring of a (derived) stack.


An auxiliary result for the proof of Theorem \ref{thm:mupower}, which is interesting on its own, is that using the power structure on $K_0(gdg\mhyphen cat)$ the categorical zeta function of a geometric dg category $\CM$ can be expressed as the power of an infinite product with exponent the category itself.
\begin{theorem}[{Corollary~\ref{cor:zcatprod}}]
\label{thm:zetaexppow}
\[Z_{cat}(\CM,t)=\prod_{n=1}^{\infty}\left( \frac{1}{1-t^n}\right)^{[\CM]}\]
where the power structure on $K_0(gdg\mhyphen cat)$ is the one considered Theorem \ref{thm:mupower}.
\end{theorem}
Another implication of Theorem~\ref{thm:mupower} is the following.
\begin{corollary}[Corollary~\ref{cor:lexchange}]
\label{cor:lambdamorph}
\begin{enumerate}
\item Replace the 2-$\lambda$-ring structure on $K_0(gdg\mhyphen cat)$ with a 1-$\lambda$-ring structure using \eqref{eq:powcat}. Then $\phi$ is a morphism of 1-$\lambda$-rings.
\item Replace the 1-$\lambda$-ring structure on $K_0(Var)$ with a 2-$\lambda$-ring structure using \eqref{eq:2lfrom1l}. Then $\phi$ is a morphism of 2-$\lambda$-rings.
\end{enumerate}
\end{corollary}

Besides the ones above, we consider a couple of further applications in the final part of the paper. Our results imply several exponential type identities in $K_0(gdg\mhyphen cat)$ including ones which are pushed forward from $K_0(Var)$. For example, we obtain a formula for the generating series of the dg categories of the Hilbert scheme of points associated with a fixed smooth projective variety. We also introduce categorical Adams operations, and calculate a formula relating them to the pushforward of the motivic Adams operations. Finally, we deduce a formula for series with exponent a linear algebraic group. 



The structure of the paper is the following. In Section~\ref{sec:powstr}, after recalling the definition of a power structure, we prove Theorems~\ref{thm:mupower},~\ref{thm:zetaexppow}, Corollary~\ref{cor:lambdamorph}. The implications for Hilbert schemes of points, Adams operations and linear algebraic groups are given in Section~\ref{sec:appl}.

\section{Power structures}
\label{sec:powstr}

\subsection{Power structures and pre-$\lambda$-ring structures} 
\label{subsec:powstrprel}

In the rest of the paper let $t$ be a formal variable.
\begin{definition}[\cite{gusein2004power}]
	\label{def:powerstr}
	A \emph{power structure} over a (semi)ring $R$ is a map 
	\[(1+tR[[t]]) \times R \to 1+tR[[t]]\colon (A(t),m) \mapsto (A(t))^m\] 
	such that:
	\begin{enumerate}
		\item\label{it:powerstr1} $(A(t))^0=1$;
		\item\label{it:powerstr2} $(A(t))^1=A(t)$;
		\item\label{it:powerstr3} $(A(t) \cdot B(t))^m=(A(t))^m \cdot (B(t))^m$;
		\item\label{it:powerstr4} $(A(t))^{m+n}=(A(t))^m\cdot (A(t))^n$;
		\item\label{it:powerstr5} $(A(t))^{mn}=((A(t))^m)^n$;
		\item\label{it:powerstr6} $(1+t)^m=1+mt+$ terms of higher degree;
		\item\label{it:powerstr7} $(A(t^k))^m=(A(t))^m\big|_{t\mapsto t^k}$.
	\end{enumerate}
\end{definition}

A power structure is called finitely determined if for each $N > 0$ there exists $M > 0$ such that the $N$-jet of the series $(A(t))^m$
(i.e. $(A(t))^{m}\; \mathrm{mod}\; t^{N+1})$ is determined by the $M$-jet of the series $A(t)$. It is possible to see that one can take $M = N$. 

The notion of a power structure over a ring is closely related with the
notion of a pre-$\lambda$-ring structure. A pre-$\lambda$-ring structure
is an additive-to-multiplicative homomorphism $R \to 1 + tR[[t]]$, $a \mapsto \lambda_a(t)$ ($\lambda_{a+b}(t) = \lambda_a(t) \cdot \lambda_b(t)$) such that
$\lambda_a(t) = 1 + at+\dots$. 
\begin{lemma} 
\label{lem:lambdafactor}
In a pre-$\lambda$-ring $R$ any series $A(t) \in 1+tR[[t]]$ can be in a unique way represented as $A(t)=\prod_{n=1}^{\infty} \lambda_{r_n}(t^n)$, for some $r_n \in R$. 
\end{lemma}
\begin{proof} Suppose that after the constant term the first nonzero coefficient in the expansion of $A(t)$ is $a_n$. That is, $A(t)=1+a_nt^n+a_{n+1}t^{n+1}+\dots+a_{2n}t^{2n}+\dots$.
Using the properties of a pre-$\lambda$-ring we write
\[ A(t)=\lambda_{a_n}(t^n)\cdot\underbrace{\left(1+b_{n+1}t^{n+1}+\dots\right)}_{=:B(t)}.  \]
Then we do the same with $B(t)$, in which the first nonzero coefficient after the constant term is in degree $n+1$.
Eventually, $A(t)$ can be written as a product of the form $\prod_{n=1}^{\infty} \lambda_{r_n}(t^n)$ with $r_n \in R$.
\end{proof}
\begin{corollary}[\cite{gusein2006power, gusein2019grothendieck}]
\label{cor:preLpower}
A pre-$\lambda$-ring structure on a ring defines a finitely determined power structure over it as follows. Writing $A(t)=\prod_{i=1}^{\infty} \lambda_{r_i}(t^i)$, its $m$-th power is given by the formula:
\[(A(t))^m\coloneqq \prod_{i=1}^{\infty} \lambda_{mr_i}(t^i).\]
This correspondence gives a surjective map from the set of
pre-$\lambda$-structures to the set of finitely determined power structures. The preimage of a power structure consists of all pre-$\lambda$-structures given by the formula
$\lambda_a(t) = (\lambda_1(t))^a$ with an arbitrary $\lambda_1(t) = 1 + t + \sum_{i=2}^{\infty}a_i t^i$.
\end{corollary}

\begin{example}  
\label{ex:1lring}	
A 1-$\lambda$-ring is a commutative unital ring $R$ equipped with operations $\Sym^k$ for every $k \geq 0$. These operations satisfy the following properties for every $m,n \in R$ and $k \geq 0$:
	\begin{enumerate}
		\item\label{it:2lrprop1} $\Sym^0(m)=1$;
		\item\label{it:2lrprop2} $\Sym^1(m)=m$;
		\item\label{it:2lrprop3} $\Sym^k(m+n)=\sum_{l=0}^k \Sym^l(m) \cdot \Sym^{k-l}(n)$;
		\item\label{it:2lrprop4} $\Sym^k(1)=1$.
	\end{enumerate}	
	The zeta function of an element $m \in R$ in a 1-$\lambda$-ring is the generating function of the symmetric powers of $m$:
	\[ Z(m,t)=\sum_{m=0}^{\infty}\Sym^k(m)t^k \in 1+tR[[t]].  \]
	A 1-$\lambda$-ring is a pre-$\lambda$-ring with $\lambda^1_m(t)=Z(m,t)$; the superscript 1 indicates that the structure comes from 1-$\lambda$-operations. By Corollary~\ref{cor:preLpower}, there is a corresponding power structure on $R$ such that 
	\begin{equation} 
	\label{eq:Zmtzeta}	
	Z(m,t)=(Z(1,t))^m=\left(\sum_{k=0}^{\infty}\Sym^k(1)\right)^m=(1+t+t^1+\dots)^m=\left(\frac{1}{1-t}\right)^m \end{equation}
In particular, $\lambda^1_m(t)=(1-t)^{-m}$.
\end{example}


\begin{example}[{\cite[Theorem 2]{gusein2004power}}] 
\label{ex:K0var}	
For a quasi-projective variety $M$, the symmetric powers $\Sym^n(M)=M^n/S_n$ are defined in the usual way. These descend to 1-$\lambda$-operations on the Grothendieck ring $K_0(Var)$ of quasi-projective varieties. By Example~\ref{ex:1lring},
\begin{equation}
\label{eq:zmotprod}
\left(\frac{1}{1-t}\right)^{[M]}=\sum_{k=0}^{\infty} [\Sym^k(M)]t^n =:Z_{mot}(M,t),\end{equation}
the motivic zeta function of $M$.  
\end{example}

\begin{example}[\cite{ganter2014symmetric}]
A 2-$\lambda$-ring has the same properties as a 1-$\lambda$-ring from Example~\ref{ex:1lring}, except that property~\ref{it:2lrprop4} is replaced by the identity
\[ \Sym^k(1)=p(k) \]
for all $k \in \mathbb{N}$ where $p(k)$ is the number of partitions of $k$.
Again, the zeta function of an element $m \in R$ in a 2-$\lambda$-ring is the generating function of the symmetric powers of $m$:
\[ Z(m,t)=\sum_{k=0}^{\infty}\Sym^k(m)t^k \in 1+tR[[t]].  \]
It defines a pre-$\lambda$-ring structure as $\lambda^2_m(t)=Z(m,t)$, which in turn induces a power structure by Corollary~\ref{cor:preLpower}; the superscript 2 indicates the 2-$\lambda$-ring property. 
\end{example}

The zeta function of the multiplicative unit of any 2-$\lambda$-ring is
\begin{equation} \label{eq:zetaunit} \lambda_1^2(t)=Z(1,t)=\sum_{k=0}p(k)t^k=\prod_{n=1}^{\infty}\frac{1}{1-t^n}\end{equation}
where we adopted the convention that $p(0)=1$. 
Because the power structure is induced by a pre-$\lambda$-ring structure, the identity $\lambda^2_m(t)=(\lambda^2_1(t))^m$ must hold for every $m \in R$ (see Corollary~\ref{cor:preLpower}). This implies the following expression for the zeta function of an arbitrary element.
\begin{corollary}
	\label{cor:2lambdazeta}
	In a 2-$\lambda$-ring, for any $m \in R$
	\[ Z(m,t)=Z(1,t)^m=\left(\prod_{n=1}^{\infty}\frac{1}{1-t^n}\right)^m.\]
\end{corollary}

 \subsection{Categorical symmetric powers}
	\label{subsec:sympow}
	
	Let $G$ be a finite group. A 2-representation (\cite[Section 4]{ganter2008representation}) of $G$ on a linear category $\CV$ consists of a collection $(\rho(g))_{g \in G}$ of linear functors $\rho(g) \colon \CV \to \CV$ as well as an isomorphism of functors
	\[ \phi_{g,h}\colon (\rho(g) \circ \rho(h)) \xRightarrow{\cong}{} \rho(gh) \]
	for every pair of elements $(g,h)$ of a $G$ and an isomorphism of functors
	\[  \phi_1 \colon \rho(1)  \xRightarrow{\cong}{} \mathrm{Id}_{\CV}.\]
	These have to satisfy  the identities
	\[ \begin{gathered}\phi_{(gh,k)}(\phi_{g,h} \circ \rho(k))=\phi_{(g,hk)}(\rho(g) \circ \phi_{h,k}), \\ \quad \phi_{1,g}=\phi_1 \circ \rho(g) \quad \textrm{and} \quad \phi_{g,1}=\rho(g) \circ \phi_1
	\end{gathered}\]
	for any $g,h,k \in G$.
	
	\begin{definition}
		Suppose we are given a 2-representation of $G$ on $\CV$. A $G$-equivariant object of $\CV$ is a pair $(A, (\epsilon_g)_{g \in G})$ where $A \in \mathrm{Ob}(\CV)$ and $(\epsilon_g)_{g \in G}$ is a family isomorphisms
		\[ \epsilon_g\colon A \xrightarrow{\sim} \rho(g)A \]
		satisfying the following compatibility conditions:
		\begin{enumerate}
			\item for $g=1$ we have
			\[\epsilon_1= \phi^{-1}_{1,A}\colon A \mapsto \rho(1)A;\]
			\item for any $g,h \in G$ the diagram
			\[
			\begin{tikzcd}[column sep = large, row sep=large]
				A \arrow[r,"\epsilon_g"] \arrow[d,"\epsilon_{gh}"] & \rho(g)(A) \arrow[d,"\rho(g)(\epsilon_h)"]\\
				\rho(gh)(A) & \rho(g)(\rho(h)(A)) \arrow[l,swap, "\phi_{g,h,A}"]
			\end{tikzcd}
			\]
			is commutative.
		\end{enumerate}
		A morphism of equivariant objects from $(A, (\epsilon_g)_{g \in G})$ and $(B, (\eta_g)_{g \in G})$ is a morphism $f\colon A \to B$ commuting with the $G$-action. 
		
		The category of $G$-equivariant objects in $\CV$ is denoted as $\CV^G$.
	\end{definition}

	\begin{example}
		\label{ex:quotcat}
		Let $M$ be a scheme of finite type over $k$.
		\begin{enumerate} 
			\item Let $\mathcal{O}_M\mhyphen \mathrm{mod}$ be the abelian category of all sheaves of $\mathcal{O}_M$-modules. Let $I(\mathcal{O}_M\mhyphen \mathrm{mod})$ be the full dg subcategory of $C(\mathcal{O}_M\mhyphen \mathrm{mod})$ consisting of h-injective complexes of injective objects. 
			Then $I(\mathcal{O}_M\mhyphen \mathrm{mod})$ is pretriangulated, and the composition $\epsilon_{I(\mathcal{O}_M\mhyphen \mathrm{mod})} \colon H^0(I(\mathcal{O}_M\mhyphen \mathrm{mod})) \to H^0(C(\mathcal{O}_M\mhyphen \mathrm{mod})) \to D(\mathcal{O}_M\mhyphen \mathrm{mod})$ is an equivalence \cite[2.1.1]{schnurer2015six} (see also \cite[Example 5.2]{bergh2016geometricity} and \cite[Thm. 14.3.1 (iii)]{kashiwara2005categories}). Therefore, $I(\mathcal{O}_M\mhyphen \mathrm{mod})$ is a dg enhancement of $D(\mathcal{O}_M\mhyphen \mathrm{mod})$.
			\item The subcategory $D(M) \coloneqq D_{pf}(M) \subset D(\mathcal{O}_M\mhyphen \mathrm{mod})$ consisting of perfect complexes is thick. Its enhancement is $I(M) \coloneqq I_{pf}(M) \subset I_{qc}(M)$. When $M$ is smooth, then $D(M)=D^b(\mathrm{Coh}(M))$, where $\mathrm{Coh}(M)$ is the category of coherent $\mathcal{O}_M$-modules.
		
		 \item Suppose $M$ is equipped with an action of a finite group $G$. Let $I(M)^G$ be the dg category of $G$-equivariant objects in $I(M)$.
		Since $I(M)$ is strongly pretriangulated, so is $I(M)^G$. 
		The composition $\epsilon_{I(M)^G} \colon H^0(I(M)^G) \to H^0(C(\mathcal{O}_M\mhyphen \mathrm{mod})^G) \to D(M)^G$ is an equivalence. By \cite[Proposition 4.3]{sosna2012linearisations} there is a dg equivalence $I(M)^G \cong I([M/G])$, which induces a triangulated equivalence $D(M)^G=D([M/G])$ (see also \cite[Theorem 9.6]{elagin2011cohomological}).
	\end{enumerate}	
\end{example}


	Let now $\CV$ be a pretriangulated category, and let $\CV^{n}$ be its  $n$-th (completed) tensor power \cite{bondal2004grothendieck}. Let the symmetric group $S_n$ act on the (generating) objects of $\CV^{n}$ by the transformation
	\[ \sigma (A_1 \otimes \cdots \otimes  A_n)=A_{\sigma^{-1}(1)} \otimes \cdots \otimes A_{\sigma^{-1}(n)},\]
	and similarly by componentwise permutation on the $\mathrm{Hom}$ complexes. Taking the identity for all $\phi_{g,h}$'s we obtain a 2-representation of $S_n$ on $\CV^{n}$. The $n$-th symmetric power $\Sym^n(\CV)$ is defined in  \cite{ganter2014symmetric} as the category of $S_n$-equivariant objects in $\CV^{n}$: 
	\[\Sym^n(\CV)=(\CV^{n})^{S_n}.\]

	Recall the notion of a geometric dg category \cite{orlov2016smooth} (see also the earlier \cite[Remark 7.2]{bondal2004grothendieck}). Briefly, a pretriangulated dg category is geometric, if it is a dg enhancement of a semiorthogonal summand in $D^b_{coh}(X)$ for some smooth projective variety $X$. Let  $K_0(gdg\mhyphen cat)$ be the Grothendieck group of geometric dg categories with relations coming from semiorthogonal decompositions. Namely, $K_0(gdg\mhyphen cat)$ is the free abelian group generated by quasi-equivalence classes of geometric pretriangulated dg categories $\CV$ modulo the relations
	\[ [\CV]=[\CA]+[\CB] \]
	where
	\begin{enumerate}
		\item $\CA$ and $\CB$ are dg subcategories of $\CV$;
		\item $H^0(\CA)$ and $H^0(\CB)$ are admissible subcategories of $\CV$;
		\item $H^0(\CV)=\langle H^0(\CA), H^0(\CB) \rangle$ is a semiorthogonal decomposition.
	\end{enumerate}
	If $\CM_1$ and $\CM_1$ are geometric dg categories, then their (completed) tensor product $\CM_1 \boxtimes \CM_2$ is also geometric. This gives $K_0(gdg\mhyphen cat)$ the structure of a ring.
	
\begin{example}
	\label{ex:dgcat2lr}	
	As it was observed by Galkin-Shinder \cite{galkin2015zeta}, the symmetric power operations 
	descend to well defined operations on $K_0(gdg\mhyphen cat)$. These will also be denoted by $\Sym^k$ and they define a 2-$\lambda$-ring structure on $K_0(gdg\mhyphen cat)$.
The (categorical) zeta function of $\CM \in K_0(gdg\mhyphen cat)$ is
\[ Z_{cat}(\CM,t)\coloneqq\sum_{k=0}^{\infty}[\Sym^k(\CM)]t^k \in 1+ tK_0(gdg\mhyphen cat)[[t]].\]
\end{example}

It was shown in \cite[Section~7]{bondal2004grothendieck} that there exists a ring homomorphism
	\[ 
	\phi \colon K_0(Var) \to  K_0(gdg \mhyphen cat).\\
	\]
	The homomorphism $\phi$ associates  with the class of a smooth projective variety $X$  the class of the category $I(X)$. Here it is used that $K_0(Var)$ can also be presented by classes of smooth projective varieties by \cite{looijenga2000motivic, bittner2004universal}.
	
The categorical zeta function of a quasi-projective variety $M$ is the zeta function of the image $\phi(M) \in K_0(gdg\mhyphen cat)$:
\[ Z_{cat}(\phi(M),t)=\sum_{k=0}^{\infty}[\Sym^k(\phi(M))]t^k \in 1+ tK_0(gdg\mhyphen cat)[[t]].\]
In particular, the categorical zeta function of a smooth projective variety $M$ is the zeta function of the class $[I(M)] \in K_0(gdg\mhyphen cat)$
\[ Z_{cat}(I(M),t)=\sum_{k=0}^{\infty}[\Sym^k(I(M))]t^k=\sum_{{ k}=0}^{\infty}[I([M^k/S_k])]t^k \in 1+ tK_0(gdg\mhyphen cat)[[t]]\]
where at the second equality we used Example~\ref{ex:quotcat}.

As a particular case of Corollary~\ref{cor:2lambdazeta} we obtain Theorem~\ref{thm:zetaexppow}.
\begin{corollary}
	\label{cor:zcatprod}
In $K_0(gdg\mhyphen cat)$, using the 2-$\lambda$-ring structure from Example~\ref{ex:dgcat2lr},
	\[Z_{cat}(\CM,t)=\prod_{n=1}^{\infty}\left( \frac{1}{1-t^n}\right)^{[\CM]}.\]
\end{corollary}

\subsection{Correspondence between 1-$\lambda$-structures and 2-$\lambda$-structures}

By Corollary~\ref{cor:preLpower} there is a simple general description of all pre-$\lambda$-structures with the same power structure: they are in one-to-one correspondence with formal series of type $1 + t + O(t^2)$; this series gives the zeta-function of 1 under the corresponding pre-$\lambda$-structure. We now make this more explicit for 1-$\lambda$ and 2-$\lambda$-structures. Our results are known to experts; we review them for later purpose.

Let $R$ be a 2-$\lambda$-ring.
\begin{lemma} 
\label{lem:invinf}	
	For each $k \geq 0$, there is an operator $A_k: R \to R$ such that
	\[Z(m,t) \cdot \left( \sum_{k=0}^{\infty}A_k(m)t^k \right) =1 . \]
\end{lemma}
\begin{proof}
There is an inverse to the series~\eqref{eq:zetaunit} defined by the property 
\begin{equation} 
\label{eq:akdef}
\left(\sum_{k=0}p(k)t^k\right)\left(\sum_{k=0}a_kt^k \right)=1. 
\end{equation}
The coefficient $a_k$, $k\geq 0$ can recursively be expressed as a polynomial of $p(i)$, $0 \leq i \leq k$.  Let us write
\[ a_k=A_k(p(1),\dots,p(k)). \]
We extend the polynomials $A_k$ to an operation on $R$ by
\begin{equation} 
\label{eq:Akdef}
A_k(m)\coloneqq[A_k(\Sym^1(m),\dots,\Sym^k(m))]. \end{equation}
By~Corollary~\ref{cor:preLpower}, these satisfy 
	\[
	\left(\sum_{k=0}^{\infty}A_k(1)t^k\right)^m=\left(\prod_{i=0}^{\infty}\lambda_{-1}(t^i)\right)^m=\prod_{i=0}^{\infty}\lambda_{-m}(t^i)=\sum_{k=0}^{\infty}A_k(m)t^k.
	\]
Therefore,
\[\begin{aligned}Z(m,t) \cdot \left( \sum_{k=0}^{\infty}A_k(m)t^k \right) & =Z(1,t)^m \cdot \left( \sum_{k=0}^{\infty}A_k(1)t^k \right)^m\\ & = \left(\sum_{k=0}p(k)t^k\right)^m\left(\sum_{k=0}a_kt^k \right)^m=1^m=1.
\end{aligned}	
\]
\end{proof}

\begin{example} The first few terms of the inverse series are $a_0=1$, $a_1=-p(1)$ and $a_2=-p(2)+p(1)p(1)$. Specifically, for a geometric dg category $\CM$, $A_0(\CM)=[1]$, $A_1(\CM)=-[\CM]$ and $A_2(\CM)=-[\Sym^2(\CM)]+[\CM^2]$.
 \end{example}


\begin{lemma}
\label{lem:invforpt}
\[ \prod_{i=1}^{\infty} \left( \prod_{n=1}^{\infty} \frac{1}{1-(t^i)^n} \right)^{\mu(i)}=\frac{1}{1-t} \]
where $\mu$ is the Möbius function.
\end{lemma}
\begin{proof} 
\begin{gather*}
\prod_{i=1}^{\infty} \left( \prod_{n=1}^{\infty} \frac{1}{1-(t^i)^n} \right)^{\mu(i)}= \prod_{k=1}^{\infty}\prod_{i \vert k}\left(\frac{1}{1-t^k} \right)^{\mu(i)}\\
=\prod_{k=1}^{\infty}\left(\frac{1}{1-t^k} \right)^{\sum_{i \vert k}\mu(i)} 
=\frac{1}{1-t},
\end{gather*}
where at the first equality we exchanged the order of the infinite products, and at the last equality we used that
\begin{equation*}
\sum_{i \vert k}\mu(i)=
\begin{cases} 
1 & \text{if } k=1, \\
0 & \text{otherwise.} 
\end{cases}
\end{equation*}
\end{proof}


Let $R$ be a 2-$\lambda$-ring. By Corollary~\ref{cor:preLpower} the 2-$\lambda$-ring structure defines a power structure on $R$. Using Corollary~\ref{cor:preLpower} in the opposite direction together with \eqref{eq:Zmtzeta}	and Lemma~\ref{lem:invinf} we can define a 1-$\lambda$-ring structure on $R$ giving rise to the same power structure as (the coefficients of)
\begin{equation} 
\label{eq:powcat}
\lambda^1_m(t)\coloneqq(1 - t)^{-m}=
 \prod_{i=1}^{\infty} \left( \prod_{n=1}^{\infty} \frac{1}{1-(t^i)^n} \right)^{m\mu(i)}=\sum_{k\geq 0} \sum_{\underline{k}: \sum_i ik_i=k} \prod_i T_{k_i}(m)t^k 
\end{equation}
where
\begin{equation}
	\label{eq:Tkidef} 
	T_{k_i}(m)\coloneqq
	\begin{cases} 
		\Sym^{k_i}(m) & \text{if } \mu(i)=1, \\
		1 & \text{if } \mu(i)=0, \\
		A_{k_i}(m) & \text{if } \mu(i)=-1.
	\end{cases}
\end{equation}
Here at the last step we expanded the infinite product and collected the coefficients of $t^k$.
 \begin{lemma}
 \label{lem:subscriptadd}
 \[\lambda^1_{m+n}(t)=\lambda^1_m(t)\cdot \lambda^1_n(t)\]
 In particular, the coefficients of expression \eqref{eq:powcat} define a 1-$\lambda$-ring structure.
\end{lemma}
 \begin{proof}
 	By the property given in Definition~\ref{def:powerstr}~\eqref{it:powerstr4} of a power structure, $(1 - t)^{-(m+n)} = (1 - t)^{-m}(1 - t)^{-n}$.
 \end{proof}
By Lemma~\ref{lem:invforpt} we have that
\[\lambda^1_{1}(t)=\frac{1}{1-t}=1+t+t^2+\dots,\]
which is equivalent to property (4) of a 1-$\lambda$-structure.
By the properties of a 2-$\lambda$-ring, $(1-t)^{-m}=1+\Sym^{1}(m)t+\dots=1+mt+\dots$. This gives property (2). Hence, by Lemma~\ref{lem:subscriptadd} we obtain that $\lambda^1_{m}(t)$ is indeed a 1-$\lambda$-structure. Again, by Corollary~\ref{cor:preLpower} this 1-$\lambda$-structure on $R$ defines the same power structure as the original 2-$\lambda$-structure $\lambda^2_{m}(t)$.


\begin{example}
\label{ex:k0gdgcat}
Writing out definition \eqref{eq:powcat} explicitly, the 1-$\lambda$-ring structure on $K_0(gdg\mhyphen cat)$ is defined by the formula
\[
 \lambda^1_{\CM}(t)\coloneqq(1 - t)^{-\CM}= \prod_{i=1}^{\infty}Z_{cat}(\CM,t^i)^{\mu(i)}=\sum_{k\geq 0} \sum_{\underline{k}: \sum_i ik_i=k} \prod_i T_{k_i}(\CM)t^k 
\]
where $T_{k_i}$ is as in \eqref{eq:Tkidef}.
\end{example}

In the other direction, given a 1-$\lambda$-structure $\lambda^1_m(t)$ on $R$ the series
\begin{equation}
\label{eq:2lfrom1l}	
 \lambda^2_m(t)\coloneqq\prod_{k=1}^{\infty}(1-t^k)^{-m}=\prod_{k=1}^{\infty}Z(m,t^k)= \prod_{k=1}^{\infty}\lambda^1(m,t^k)\end{equation}
gives a 2-$\lambda$-structure over $R$. Indeed, say, property (3) is implied by the identity
	\[  \prod_{k=1}^{\infty}(1-t^k)^{-m-n}=\prod_{k=1}^{\infty}(1-t^k)^{-m}(1-t)^{-n}=\left( \prod_{k_1=1}^{\infty}(1-t^{k_1})^{-m} \right) \left(\prod_{k_2=1}^{\infty}(1-t^{k_2})^{-n} \right).\]
Once again, the 2-$\lambda$-structure defined this way gives the same power structure on $R$ as $\lambda^1_m(t)$.

\subsection{Compatibility of power structures}
\label{subsec:powercat}



Let $R_1$ and $R_2$ be rings. A ring homomorphism
$ \phi \colon R_1 \to R_2$ induces a natural homomorphism $R_1[[t]] \to R_2[[t]]$
(also denoted by $\phi$) by $\phi(\sum_n a_n t^n)=\sum_n \phi(a_n)t^n$.

\begin{lemma}
	\label{lem:ringhomo} Suppose that $R_1$ and $R_2$ are rings with finitely determined power structures. If the ring homomorphism $\phi \colon R_1 \to R_2$ is such that $\phi((1-t)^{-a})=\phi(1-t)^{-\phi(a)}$, then $\phi((A(t))^m)=(\phi(A(t)))^{\phi(m)}$.
\end{lemma}
\begin{proof}
	Applying~Lemma~\ref{lem:lambdafactor} in the setting of Example~\ref{ex:1lring} we have that every $A(t) \in  1+tR_1[[t]]$ can be written as a product of the form $\prod_{n=1}^{\infty}(1-t^n)^{-r_n}$ for some $r_n \in R_1$.
	By properties \eqref{it:powerstr3} and \eqref{it:powerstr7} from Definition~\ref{def:powerstr} and the finite determinacy of the power structure one
	has that
	\begin{equation} \label{eq:factpowdef} (A(t))^m= \prod_{n=1}^{\infty}(1-t^n)^{-r_nm}.\end{equation} 
	Hence, to check a property of a ring homomorphism with respect to finitely determined power structures it is enough to check it on elements of the form $(1 - t)^{-a}$, $a \in R$.  This observation implies the lemma.
\end{proof}

Suppose that $R_1$ is equipped with a 1-$\lambda$-ring structure and $R_2$ is equipped with a 2-$\lambda$-ring structure. In particular, both rings carry a finitely determined power structure. Denote by $Z_1(m,t)$ and $Z_2(m,t)$ the corresponding zeta functions. The following is our key observation. 
\begin{theorem} 
\label{thm:aux}
	Let $\phi$ be a homomorphism such that $Z_2(\phi(m),t)=\prod_{n \geq 1} \phi(Z_1(m,t))$. Then
	\[\phi((A(t))^m)=(\phi(A(t)))^{\phi(m)}.\]
\end{theorem}
\begin{proof}
	Since $\phi$ is a ring homomorphism, 
	$\phi \left( (1-t)^{-1}\right) = (1-t)^{-1}.$
	Hence, for any $M \in R_1$
	\begin{gather*}
		\phi\left( \frac{1}{1-t}\right)^{\phi(M)}=\left( \frac{1}{1-t}\right)^{\phi(M)}= \prod_{m=1}^{\infty} \left( \prod_{n=1}^{\infty} \frac{1}{1-(t^m)^n} \right)^{\mu(m) \phi(M)}\\ 
		=\prod_{m=1}^{\infty} Z_{2}(\phi(M),t^m)^{\mu(m)} = \prod_{m=1}^{\infty} \prod_{n = 1}^{\infty} \phi(Z_{1}(M,t^{mn}))^{\mu(m)}  \\=
		\phi\left(\prod_{m=1}^{\infty} \prod_{n = 1}^{\infty} Z_{1}(M,t^{mn})^{\mu(m)}\right) =\phi\left(\prod_{n=1}^{\infty} \prod_{k \vert n} Z_{1}(M,t^{k})^{\mu\left(\frac{k}{n}\right)}\right)\\ 
		=\phi\left(\prod_{n=1}^{\infty} Z_{1}(M,t^{k})^{\sum_{k \vert n}\mu(k)}\right)=\phi\left( Z_{1}(M,t)\right) 
		=\phi\left(\left( \frac{1}{1-t}\right)^M\right),
	\end{gather*}
	where at the second equality we have used Lemma~\ref{lem:invforpt}, at the third equality we have used Corollary~\ref{cor:zcatprod}, at the fourth equality we have used the assumption, 
	at the fifth equality we have used that $\phi$ is a ring homomorphism, and at the last equality we have used \eqref{eq:Zmtzeta}.
	Combining this calculation with Lemma~\ref{lem:ringhomo} gives the statement.
\end{proof}

Recall from \cite[Section 7]{bondal2004grothendieck} that there exists a ring homomorphism
\[ 
\phi \colon K_0(Var) \to  K_0(gdg \mhyphen cat).\\
\]
The homomorphism $\phi$ associates the class of the category $I(X)$ with the class of a smooth projective variety $X$ . By the above, this extends to a ring homomorphism
\[ \phi \colon 1+tK_0(Var)[[t]] \to 1+tK_0 (gdg\mhyphen cat)[[t]]. \]

Our main Theorem~\ref{thm:mupower}, evoked and proved as Theorem~\ref{thm:ringhomo2} below, shows that $\phi$ is compatible with the power structures on $K_0(Var)$ and $K_0 (gdg\mhyphen cat)$. The main ingredient of its proof is the following relation between the categorical and the motivic zeta function of a variety, which was first conjectured by Galkin and Shinder \cite{galkin2015zeta}.
\begin{theorem}[{\cite[Theorem B]{bergh2017categorical}}]
\label{thm:catmotrel}
For any quasi-projective variety $M$
\[ Z_{cat}(\phi(M),t)=\prod_{n \geq 1} \phi(Z_{mot}(M,t^n)). \]
\end{theorem}

\begin{theorem}
\label{thm:ringhomo2}
The ring homomorphism $\phi \colon 1+tK_0(Var)[[t]] \to 1+tK_0 (gdg\mhyphen cat)[[t]]$ satisfies
\[ \phi\left((A(t))^{m}\right)=(\phi(A(t)))^{\phi(m)}. \]
\end{theorem}
\begin{proof}
By Theorem~\ref{thm:catmotrel}, we can apply Thereom~\ref{thm:aux} in this setting to obtain the claim.
\end{proof}

\begin{corollary}\label{cor:lexchange}
\begin{enumerate}
\item Replace the 2-$\lambda$-ring structure on $K_0(gdg\mhyphen cat)$ with the 1-$\lambda$-ring structure using \eqref{eq:powcat}. Then $\phi$ is a morphism of 1-$\lambda$-rings:
\[ \phi(\lambda^1_M(t))=\lambda^1_{\phi(M)}(t).\] 
\item Replace the 1-$\lambda$-ring structure on $K_0(Var)$ with the 2-$\lambda$-ring structure using \eqref{eq:2lfrom1l}. Then $\phi$ is a morphism of 2-$\lambda$-rings:
\[ \phi(\lambda^2_M(t))=\lambda^2_{\phi(M)}(t).\] 
\end{enumerate}
\end{corollary}
\begin{proof} Both statements follow from Theorem~\ref{thm:ringhomo2} and the fact that both \eqref{eq:powcat} and \eqref{eq:2lfrom1l} preserve the power structure.
\end{proof}

\section{Applications}
\label{sec:appl}

\subsection{The Hilbert zeta function}


Let $M$ be any quasi-projective variety. The Hilbert zeta function is defined as the generating series of the motives of Hilbert scheme of points on $X$:
\[ H_M(t)=\sum_{n=0}^{\infty} [\mathrm{Hilb}^n(M)] t^n \; \in 1+tK_0(Var)[[t]].\]
Here $\mathrm{Hilb}^n(M)$ is the Hilbert scheme of $n$ points on $M$.
\begin{proposition}[{\cite[Theorem 1]{gusein2006power}}]
\label{prop:hilbexp}
For a smooth quasi-projective variety $M$ of dimension $d$
\[H_M(t)= \left(H_{(\SA^d,0)}(t)\right)^{[M]},\]
where $H_{(\SA^d,0)}(t)=\sum_{n=0}^{\infty}[\mathrm{Hilb}^n(\SA^d,0)]t^n$ is the generating series of the motives of Hilbert scheme of $n$ points on $\SA^d$ supported at the origin.
\end{proposition}

Theorem~\ref{thm:ringhomo2} and Proposition~\ref{prop:hilbexp} together give the following result.
\begin{corollary} For a smooth projective variety $M$ of dimension $d$,
\[\sum_{n=0}^{\infty} [I(\mathrm{Hilb}^n(M))] t^n=\phi(H_M(t))=\phi\left(H_{\SA^d,0}(t^n)\right)^{[I(M)]}. \]
\end{corollary}

\begin{example}
For $d=1$ it is known that $\mathrm{Hilb}^n(\SA,0)=\mathrm{Sym}^n(\SA,0)=\{pt\}$. Hence, $H_{(\SA,0)}(t)=(1-t)^{-1}$.
By \eqref{eq:zmotprod}, this implies that for a smooth projective curve $C$
\[\sum_{n=0}^{\infty} [I(\mathrm{Hilb}^n(C))] t^n=\left(\frac{1}{1-t}\right)^{[I(C)]}=\phi(Z_{mot}(C,t)).\]
\end{example}

\begin{example}
For $d=2$ it was shown in \cite[Theorem~1.1(iv)]{ellingsrud1987homology} (see also \cite{gusein2004power}) that
\[ H_{(\SA^2,0)}(t)=\prod_{n=1}^{\infty}\frac{1}{1-\mathbb{L}^{n-1}t^n}. \]
Here $\mathbb{L}=[\mathbb{A}^1]$, the class of the affine line.
Since $\phi(\mathbb{L})=1$, it follows that for a smooth projective surface $S$
\[\sum_{n=0}^{\infty} [I(\mathrm{Hilb}^n(S))] t^n=\prod_{n=1}^{\infty}\left(\frac{1}{1-t^n}\right)^{[I(S)]}=\prod_{n=1}^{\infty}\phi(Z_{mot}(S,t^n))=Z_{cat}(S,t).\]
\end{example}

\begin{example}
For $d=3$ the Hilbert scheme of points is not smooth any more  \cite[Page~341]{behrend2008symmetric}. It turns out \cite{behrend2013motivic} that instead of the usual Grothendieck ring of varieties it is better to work in $K_0(Var)[\mathbb{L}^{-\frac{1}{2}}]$, and the Hilbert scheme of a threefold has a virtual motive $[\mathrm{Hilb}^n(M)]_{vir} \in K_0(Var)[\mathbb{L}^{-\frac{1}{2}}]$.
The Hilbert zeta function of a quasiprojective threefold $M$ is
\[ H_{M}^{vir}(t)=\sum_{n=0}^{\infty} [\mathrm{Hilb}^n(M)]_{vir} t^n \; \in 1+tK_0(Var)[\mathbb{L}^{-\frac{1}{2}}][[t]],\]
and by \cite[Proposition~4.2]{behrend2013motivic} this possesses the property from Proposition \ref{prop:hilbexp}:
\[H_M^{vir}(t)= \left(H^{vir}_{(\SA^3,0)}(t)\right)^{[M]}.\]
The expression
\[ H^{vir}_{(\SA^3,0)}(t)=\prod_{n=1}^{\infty}\prod_{k=0}^{n-1}\frac{1}{1-\mathbb{L}^{k+2-n/2}t^n} \]
was proved in \cite[Theorem~3.7]{behrend2013motivic}.
Formally, the morphism
\[\phi\colon K_0(Var) \to K_0(gdg\mhyphen cat)\] 
can be extended to a morphism
\[\phi\colon K_0(Var)[\mathbb{L}^{-\frac{1}{2}}] \to K_0(gdg\mhyphen cat),\]
such that $\phi(\mathbb{L}^{-\frac{1}{2}})=-\phi(pt)=-1$. The justification of this extension comes from \cite[Proposition 1.15]{behrend2013motivic},
and our results carry over to this case. As a consequence, for a smooth projective threefold $M$
\[\sum_{n=0}^{\infty} \phi([\mathrm{Hilb}^n(M)]_{vir}) t^n=\prod_{n=1}^{\infty}\left(\frac{1}{1-(-t)^n}\right)^{n[I(M)]}=\prod_{n=1}^{\infty}\phi(Z_{mot}(M,(-t)^n))^n.\]
\end{example}

\subsection{Adams operations}

Following \cite{knutson2006lambda,gorsky2009adams} we introduce Adams operations for the various power structures.
The motivic Adams operations $\Psi_n^{mot}$ for a class $M \in  K_0(Var)$ are defined via
\[ Z_{mot}(M,t)^{-1} \cdot \frac{d}{dt}Z_{mot}(M,t)=\sum_{n=1} \Psi_n^{mot}(M)t^{n-1}. \]
Similarly, we define the categorical Adams operations $\Psi_n^{cat}$ for a class $\mathcal{M} \in K_0(gdg\mhyphen cat)$ as
\[ Z_{cat}(\mathcal{M},t)^{-1} \cdot \frac{d}{dt}Z_{cat}(\mathcal{M},t)=\sum_{n=1} \Psi_n^{cat}(\mathcal{M})t^{n-1}. \]


\begin{lemma}
The motivic and categorical Adams operations are related as
\[ \Psi_n^{cat}(\phi(M))= \sum_{k|n} \frac{n}{k} \phi(\Psi_k^{mot}(M)). \]
\end{lemma}
\begin{proof}
\begin{gather*}
\sum_{n=1} \Psi_n^{cat}(\phi(M))t^{n-1} = Z_{cat}(\phi(M),t)^{-1} \cdot \frac{d}{dt}Z_{cat}(\phi(M),t)\\ 
= \prod_{m \geq 0} \phi(Z_{mot}(M,t^m))^{-1} \cdot \frac{d}{dt}\prod_{m \geq 0} \phi(Z_{mot}(M,t^m)) \\
= \prod_{m \geq 0} \phi(Z_{mot}(M,t^m))^{-1} \cdot \left( \sum_{m=1}^{\infty} \frac{d}{dt} \phi(Z_{mot}(M,t^m))  \prod_{\substack{k \geq 0 \\ k \neq m}} \phi(Z_{mot}(M,t^k))\right) \\
= \prod_{m \geq 0} \phi(Z_{mot}(M,t^m))^{-1} \prod_{m \geq 0} \phi(Z_{mot}(M,t^m)) \\ \cdot \sum_{m=1}^{\infty} \phi\left( \frac{d}{dt} Z_{mot}(M,t^m) \right) \phi(Z_{mot}(M,t^m))^{-1}  \\
= \sum_{m=1}^{\infty}  \phi\left( Z_{mot}(M,t^m)^{-1} \frac{d}{dt} Z_{mot}(M,t^m)  \right)\\
=  \sum_{m=1}^{\infty}  \phi\left(  mt^{m-1} \left(Z_{mot}(M,t)^{-1}\frac{d}{dt}Z_{mot}(M,t)\right)\Bigg|_{t \mapsto t^m}  \right) \\
= \sum_{m=1}^{\infty} \phi\left( mt^{m-1} \cdot \sum_{k}\Psi_k^{mot}(M)t^{m(k-1)} \right) = \sum_{m=1}^{\infty} \sum_{k} m\phi(\Psi_k^{mot}(M))t^{mk-1}
\end{gather*}
where at the second equality we used Theorem \ref{thm:catmotrel}, at the fourth equality we used that $\phi$ is a ring homomorphism, and at the fifth equality, besides cancellation, we used Theorem \ref{thm:mupower}.
\end{proof}

\begin{example}
\[ \Psi_n^{cat}(1)= \sum_{k|n} \frac{n}{k} \phi(\Psi_k^{mot}(1))=\sum_{k|n} \frac{n}{k}\phi(1)=\sum_{k|n} k = \sigma(n), \]
the sum of divisors of $n$.
\end{example}

\begin{lemma}[{\cite[Chapter 1, 2.14 and 2.15]{macdonald1998symmetric}}] After tensoring with $\mathbb{Q}$, the following identities relate the zeta functions and the corresponding Adams operations:
\[Z_{mot}(M,t)=\mathrm{exp}\left( \sum_{n \geq 1}\Psi^{mot}_n(M) \cdot \frac{t^n}{n}\right) \in 1+tK_0(Var) \otimes\mathbb{Q}[[t]], \]
\[Z_{cat}(\mathcal{M},t)=\mathrm{exp}\left( \sum_{n \geq 1} \Psi^{cat}_n(\mathcal{M}) \cdot \frac{t^n}{n}\right) \in 1+tK_0(gdg\mhyphen cat) \otimes \mathbb{Q}[[t]] .\]
\end{lemma}

\begin{corollary} In $1+tK_0(gdg\mhyphen cat) \otimes \mathbb{Q}[[t]]$,
\[ Z_{cat}(\phi(M),t) =\mathrm{exp}\left(\sum_{n \geq 1} \sum_{k |n} \phi(\Psi^{mot}_n(M)) \cdot \frac{t^n}{k}\right).   \]
\end{corollary}

\subsection{Linear algebraic groups}
Denote by $G^0$ the identity component of a linear algebraic group $G$. Recall that the rank $\mathrm{rk}(G) $ of $G$ is the dimension of a maximal torus $T \subset G$. 
\begin{corollary} Let $G$ be a linear algebraic group.
Then the image under $\phi$ of any series in $K_0(Var)$ with exponent $[G]$ is
\[\phi\left((A(t))^{[G]}\right) = 
\begin{cases}
1,  & \textrm{if }\; \mathrm{rk}(G) > 0, \\
\phi(A(t))^{|G:G^0|},   & \textrm{if }\; \mathrm{rk}(G) = 0
\end{cases}\]
 in $1+tK_0(gdg\mhyphen cat)[[t]]$.
In particular, the categorical zeta function of a linear algebraic group of positive rank is 1.
\end{corollary}
\begin{proof}
By Theorem \ref{thm:mupower},
\[ \phi\left((A(t))^{[G]}\right)=(\phi(A(t)))^{\phi(G)}. \]
We will show that $\phi(G)=0$ if $\mathrm{rk}(G) > 0$, and $\phi(G)=|G:G^0|$ otherwise. The components of an algebraic group $G$ are isomorphic as varieties. 
We therefore assume that $G$ is connected (and show that $\phi(G)=1$, if the rank is 0) . Let $U$ denote the unipotent radical of $G$ and $H = G/U$. Then $U$ is isomorphic to a subgroup of the unitriangular matrices $n \times n$ matrices \cite[Section 4.8]{borel2012linear}. The logarithm map defines an isomorphism of varieties between $U$ and its Lie algebra, which means that $U$ is isomorphic to $\mathbb{A}^{\mathrm{dim} U}$. The existence of Levi decompositions gives an isomorphism $G\cong U \times H$. Let $B$ denote a Borel subgroup of $H$, $T$ a maximal torus of $B$, and $V$ the
unipotent radical of $B$. Decomposition into Schubert cells gives a stratification of the flag variety $H/B$.  Every stratum of this stratification is isomorphic to $\mathbb{A}^r$ for some $r$ with the preimage $VwB \subset  H$ isomorphic to $\mathbb{A}^r \times B \cong \mathbb{A}^{r+\mathrm{dim} V} \times T$ \cite[Section 14.12]{borel2012linear}. Hence, the class $[G] \in K_0(Var)$ equals $[T]$ times the class of an affine space. As \[\phi(T)=\phi(\mathbb{G}_m^{\mathrm{rk}(G)})=(\phi(\mathbb{A}^1)-\phi(pt))^{\mathrm{rk}(G)}=(1-1)^{\mathrm{rk}(G)}=0^{\mathrm{rk}(G)},\]
the result follows.
\end{proof}

\subsection*{Acknowledgement} The author would like to thank to Jim Bryan, Sabin Cautis, Eugene Gorsky, Sándor Kovács, Daniel Litt, Roberto Pirisi, and Zinovy Reichstein for helpful comments and discussions. A part of this work was carried out while the author was at the Department of Mathematics, University of British Columbia, Canada. The author was supported by the János Bolyai Research Scholarship of the Hungarian Academy of Sciences. Competing interests: The author declares none.

\bibliographystyle{abbrv}
\bibliography{powerstr}


\end{document}